\documentclass[letterpaper, 10 pt, conference]{ieeeconf}  

\IEEEoverridecommandlockouts                              
\def\d{\mathrm{d}}

\def\e{\mathrm{e}}
\def\R{\mathbb{R}}
\def\X{\mathcal{X}}
\def\contract{\rfloor}
\def\var{{\rm{Var}}}

\newtheorem{theorem}{Theorem}

\usepackage{graphicx} 
\usepackage{amsmath} 
\usepackage{amssymb}  

\usepackage{todonotes}
\usepackage[utf8]{inputenc}
\usepackage{multicol}
\usepackage{dblfloatfix}

\title{\LARGE \bf
The Barycenter Method for Direct Optimization: an Overview
}

\author{Felipe~Pait
\thanks{Universidade de S\~{a}o Paulo, Brazil; {\tt\footnotesize pait@usp.br}.}%
}

\begin{document}


\maketitle
\thispagestyle{empty}
\pagestyle{empty}


\begin{abstract} A randomized version of the recently developed barycenter method for derivative--free optimization  has desirable properties of a gradient search. We developed a complex version to avoid evaluations at high--gradient points. The method, which is also applicable to non--convex and to non--smooth functions, is parallelizable in a natural way and shown to be robust under noisy measurements. The goal of this paper is to present an overview of the method, whose properties  make it particularly useful in control applications. 
\end{abstract}

\section{INTRODUCTION}

The recently developed barycenter method for direct optimization has properties that make it particularly promising in control applications. These properties include: 
\begin{enumerate}
  \item The barycenter method is a form of derivative--free optimization: it aims at finding extremal points of a function whose  mathematical expression is not precisely known. In Section~\ref{sec:derivative-free} we give an overview of derivative--free optimization. We believe that the informal discussion, although somewhat long--winded, can be helpful as an introduction to the field and as an argument in favor of a direct optimization method which has a grounding in closed--loop controls thinking.
  \item The method can be described both in a recursive, algorithmic manner, and in a closed form or batch expression which is suitable for mathematical analysis. These expressions, given in Section~\ref{sec:recurse-and-batch}, provide the method with a transparent interpretation, in contrast with the ad--hoc justifications of many other direct search procedures.
  \item The method is compatible with  other well--established search procedures, and thus can make use of previous understanding of what works in a given situation. On the other hand, if no a priori knowledge is available, a purely random local search strategy may be employed. In this most unfavorable of situations, the barycenter method behaves in a gradient--descent manner. This is the content of Theorem~\ref{thm:average}, proved in Section~\ref{sec:theorems}.
  \item Theorem~\ref{thm:variance} in Section~\ref{sec:further} is a formula for the variance of the estimate when a random search is employed, and Theorem~\ref{thm:complex-interference}  contains an important property of the complex version of the barycenter method. It is our belief that the complex version, which is yet to be more fully tested, can behave in a most desirable way for a variety of applications, including to controls.
  \item The barycenter method is intrinsically applicable to non--differentiable functions and it is robust to measurement noise, as shown in Theorem~\ref{thm:under-noise},  stated in Section~\ref{sec:further}.
  \item The barycenter search is performed in convex sets: the convex hull of all test points. The method is applicable to non--convex functions, although inevitably the existence of local minima   presents a problem for any global optimization scheme. 
\end{enumerate}

Proofs of Theorems~\ref{thm:variance}, \ref{thm:complex-interference}, and \ref{thm:under-noise}, stated in Section~\ref{sec:further}, are available in the arXiv \cite{2018bary-arXiv}. Simulations are presented and interpreted in Section~\ref{sec:simula}. Results,  applications, and directions for further work, specially applications to direct and combined direct--indirect adaptive control, are discussed in Section~\ref{sec:discussion}.

\section{DIRECT OPTIMIZATION IN CONTEXT}\label{sec:derivative-free}

The barycenter method, which will be expressed by formula~\eqref{eq:barycenter} in \ref{sec:recurse-and-batch}, is suitable for {direct optimization, once scorned} \cite{once-scorned-wright}, now a respectable research area, both for its scientific challenges and practical applications. Also known as  \textsl{derivative--free optimization}, it deals with the search for extrema of a given function, employing only the values of the function and not its mathematical expression. 
When the 1st and perhaps also 2nd derivatives of the function are available, use of gradients and Hessians leads to steepest--descent and Newton--like search algorithms. Often however derivatives are costly or impossible to compute. The challenge in direct optimization is to obtain algorithms with comparable performance, without knowledge of the derivatives. 

The barycenter method has been employed successfully to tune filters  in system identification \cite{MOLI-transactions}, see also \cite{RomanoCDC:2014}. 
The filter parameters are additional structure parameters that need to be chosen before the model parameters themselves can be optimized, and may  be said to perform a role similar to that of hyper--parameters in machine learning. With that experience in mind, the use of the barycenter method  to tune hyper--parameters seems worth exploring.
A continuous--time version of the barycenter algorithm was analyzed in \cite{pait2014reading}. Some aspects of the method and its applications were presented at SIAM conferences \cite{SiamSDiego2014,SiamBoston2016}. 

The material that follows is a discussion of derivative--free optimization in the context of the  more familiar methods of nonlinear programming, and  of the barycenter method in the context of derivative--free optimization, always with a point of view that is appropriate for the kinds of problems that may appear in control theory and applications.  Although the work discussed in this paper is essentially on static, nonlinear mathematical optimization theory, we judge that the method has potential for applications  to closed--loop control of dynamical  systems.
The recent books  \cite{derivative-free-book,audet2017derivative} provide more complete expositions of derivative--free optimization and serve as good entry points to the literature. 
Section~\ref{sec:derivative-free} is meant as an informal, intuitive digression, and can be skipped by the reader interested in the mathematical properties of the barycenter method.

\subsection{Nonlinear optimization }
Nonlinear programming methods usually start from the assumption that a mathematical expression for the function being minimized exists. This is often not the case. In many problems, the physical or other scientific principles behind the problem suggest general properties of the function $f$, but do not supply exact values. Perhaps determining the function would require extensive modeling work that is costly or impossible to perform. Perhaps the function changes over time with the conditions of operation of some machinery, or its precise values are subject to some uncertainty or corrupted by measurement noise. In this case we may need to avail ourselves of some method for direct, also called derivative--free, optimization.

What we have is an oracle: an experiment, or perhaps a computer simulation, which will supply the value of $f(x_i)$ at a point $x_i$ whenever questioned. Each oracle query has a cost: performing the experiment will need human labor, a simulation will take computer time. Also perhaps the act of trying out a certain parameter guess will have an effect on the process being studied, which may be undesirable if the value of $f$ is high. We have in mind particularly the case of real--time decisions. For example suppose we are trying to optimize the parameters of a feedback process control device. In the time during which we tried out a ``bad'' controller, the performance of the process was poor. Now we will have learned to avoid that particular set of parameters, at the cost of allowing for poor control during a certain interval. The criteria by which we judge a direct optimization method may be different from the usual criteria for numerical optimization, which focus exclusively on computational complexity and demands, rather than on the values attained by the objective function during the search procedure.

\subsection{Derivative--free optimization}

One way to approach the problem would be to estimate the function $f$ using a sequence of oracle test queries. Whether this is promising depends crucially on our previous knowledge about the shape of the function and its properties. If the function can be described with a small number of parameters, and those can be computed on the basis of a limited number of polled values, then such a method may be quite effective. After a mathematical expression is derived, optimization may proceed very much along the lines of optimization algorithms for known functions.

If however our prior knowledge is limited, and the objective needs to be expressed in terms of a general--purpose functional approximation method, then the amount of data necessary to obtain a reliable functional expression can become excessive. More often than not, the number of experiments or simulations needed to obtain a functional approximation will become larger than what is necessary to simply find a minimum. Among the most successful methods that try to combine functional approximation with search for optimizers is Bayesian learning, which is described in an extensive literature.


If one wants to avoid the extra work of estimating the complete function $f$ using oracle queries, one approach is to emulate the nonlinear programming methods, using the function values to estimate derivatives. Derivatives are however notoriously difficult to measure using the differences between the values of the function in neighboring points. If the points are too far, then what is being estimated is not the value of the derivative; if too close, then the difference is dominated by noise or numerical errors. As we know from our control theory and experience, taking derivatives of measured signals rarely gives solid results.

With this in mind, many algorithms for derivative-free optimization have been studied in the literature. Those include methods based on  biological or evolutionary analogies for the exploration of the search space; methods using hypercubes or simplexes to bound regions where the minima may be contained; search in the direction of coordinates or using other patterns; trust--region methods using local linear or quadratic approximations to the functions; and others.

\subsection{Methods for derivative--free optimization}

Let's revisit some usual direct optimization techniques to try to understand where the barycenter method fits. Recall that the majority of the literature in mathematical optimization deals with problems where the goal to be optimized is formulated as a well--defined mathematical function, of which we can compute derivatives and second derivatives, and perform other useful operations. It is often not the case that the expressions are available. Then we are forced to look for minima (or maxima) using measurements of the value of the goal.
%
%
Derivative--free optimization often proceeds along the following lines:
\begin{itemize}
\item Imitate conventional nonlinear programming methods, approximating derivatives using measured values.
\item Estimate the function itself using sampled measurements, and then use various methods on the estimated function as if we were certain that the estimate was correct.
\item Specify a recursive algorithm for generating a sequence of guesses of where the minimum might be located, and endeavor to show that the procedure converges to an acceptable answer to the minimization question.
\end{itemize}

Let us consider each of these alternatives. It need not be argued to an audience familiar with automatic control that finite--differences approximations to derivatives don’t often result in useful techniques. Computing derivatives experimentally is not computationally or algorithmically difficult; it is simply useless. The realities of measurement noise and procedural errors will make any approximation prone to errors and in need of extensive filtering, which requires repetitive sampling, and increases the number of measurements needed.

Functional estimates are often based on explicit parameterizations, which may be global or use a patchwork of local functions defined within trust regions. Local parameterizations may be based on linear, quadratic, sigmoidal, Gaussian, or fuzzy--logic type basis functions. Global parameterizations could use familiar power or sinusoidal series. Some of the most successful methods that consider global parameterizations are the Bayesian learning ones. There also exist so-called “nonparametric models”, where the parameterizations are given implicitly by the assumptions made in generating the estimates and the algorithm that produces the sequence of probes, the dimension of the parameter space being very large, perhaps uncountably so.

Effectiveness of the procedure depends on how well the model class matches the class of functions being estimated. If we have workable prior knowledge about the general shape or mathematical expression of the function being optimized, then each sample will contribute to the knowledge of its overall behavior, and learning of the function itself can proceed economically. On the other had if the goal function does not match well the prior assumptions embedded into the parameterization, learning its overall shape will require extensive probing across all of its domain. It becomes likely that this procedure will be overwork if our goal is simply to find a minimum value of the function, rather than a precise global model for all of its values.

Search algorithms for derivative--free optimization have been engineered using a variety of different principles. Some use coordinate patterns to pick the search directions, or geometrical constructs such as hypercubes or other polytopes to progressively shrink the portion of the search space under consideration. In the Nelder--Mead method, one of the oldest and best known in direct optimization, the polytopes are triangles, tetrahedra, or higher--dimensional simplexes. Another category of algorithms are inspired by the individual or collective behavior of animals or biological phenomena such as recombination of genetic material, mutation, and evolution.

Many of the techniques make use of randomization, and even (almost completely) random searches have been considered. All of them provide good ideas and intuition to help specify the algorithms, however the narrative of how they operate often sounds like a just--so story. It is often not clear that Nature runs those procedures with the goal of finding minima and maxima, nor does the motivation furnish convincing arguments for why the extrema will be reached by the procedure. Proving convergence is a fruit of labor, and may require exacting assumptions on the goal functions.


Its is partly in response to the considerations above that the barycenter method has been developed. Its analysis follows from the equivalence between the batch and recursive formulations in \ref{sec:recurse-and-batch}. While the recursive version is flexible enough to incorporate any search mechanism that is found to have merits for a particular problem, the batch formula provides a closed expression that serves to analyze the properties of the method, in particular its robustness to noise and non--differentiability.

One of the desirable properties of gradient--based methods is that they incorporate naturally the concept that points where the derivative is large should be skipped over quickly, at least in the case of differentiable goal functions. This is not an easy notion to consider in most of the derivative--free methods. The complex version of the barycenter method has the property that points with large values of the derivative are given a low weighting. If they cannot be avoided without explicitly taking derivatives, at least the method can be set up so that tests at high-derivative points will not lead to a big waste of resources during the course of the search.

\subsection{The barycenter method}


The barycenter method has a very straightforward rationale. Search points are given an exponential weight, which is large for points where the function has low values and small where the goal function is small. The search points are combined to produce an estimate for the minimum, deemed to be  at their center of mass.

The method has equivalent batch and recursive formulations. The equivalence of the formulations provides an algorithmic approach and facilitates mathematical analysis of the properties.

The barycenter method may be used to combine a sequence of test points independently of how the sequence was generated. It is not in opposition to the search algorithms previously studied in the literature --- any of them, or more than one of them, can be used to generate the test points. In the same way that different approaches can be combined, the barycenter method is naturally parallelizable.

If no previous knowledge exists, or if one decides not to use the existing methods, then it is reasonable to use a purely random search to generate the sequence of test points recursively. In this case it can be shown that the barycenter method produces a sequence of steps that follows a gradient descent pattern, without the need to compute or estimate gradients. Theorem~\ref{thm:average} gives a proof of this statement. Theorem~\ref{thm:variance} shows that the variance of the step size is reduced with respect to the variance of the randomized search, which is an indication of the method’s robustness.

The barycenter method has a complex version which is inspired by Feynman's interpretation of quantum mechanics. The advantage of the complex version is that it avoids and discounts tests made at points where the derivative of the goal function is high. Recall that such points do not fulfill the necessary conditions for minimality, however they cannot be a priori excluded from the search if we are restricted to using derivative--free methods. Using a form of destructive interference between nearby points, the complex version of the barycenter method goes for the second best option, which consists in giving lower weight to measurements made at high derivative points, as shown in Theorem~\ref{thm:complex-interference}.

Another advantage is that  the method is by construction tolerant to noise, and to measurement and numerical errors. Theorem~\ref{thm:under-noise} provides approximations for the mean bias and variance of the estimate of the minimum introduced by the presence of noise. The analysis of the method’s behavior in the presence of noise is facilitated by the simple expressions for its recursive and batch versions, and by their straightforward equivalence. Besides easing the mathematical analysis, the simplicity of the method makes coding and implementation more transparent.
\section{THE BARYCENTER METHOD: RECURSIVE AND BATCH EXPRESSIONS} \label{sec:recurse-and-batch}

The barycenter method consists of searching for a minimizer of a function $f : \X \rightarrow \R$ using the formula
\begin{equation}
\label{eq:barycenter}
\hat{x}_n = \frac{\sum_{i=1}^n x_i \e^{- \nu f(x_i)}}{\sum_{i=1}^n \e^{- \nu f(x_i)}},
\end{equation}
which expresses the  center of mass of $n$ test points $x_i  \in \X \subset \R^{n_x}$ weighted according to the exponential of the  value of the function at each point. In this formula $\nu \in \R$ is a positive constant. The point $\hat{x}_n$ will be in the convex hull of the $\{x_i\}$, and we shall assume that the $n_x$--dimensional search space $\X$ is convex. The rationale behind it is that points where $f$ is large receive low weight in comparison with those for which $f$ is small. 

The equation above is equivalent to  the recursive formulas
\begin{align}
m_n & = m_{n-1} + \e^{-\nu f(x_n)} \label{eq:mass_update} \\
\hat{x}_n & = \frac{1}{m_n} \left( m_{n-1} \hat{x}_{n-1} + \e^{-\nu f(x_n)} x_n \right). \label{eq:bary_update} 
\end{align}
Here $m_0 = 0, \hat{x}_0 $ is arbitrary, and $x_n$ is the 
 sequence of test values.

From the point of view of recursive search strategies, it can be useful to pick the sequence of test points $x_n$ as the sum of the barycenter $\hat{x}_{n-1}$ of the previous points and a  ``curiosity'' or exploration term $z_n$:
\begin{equation}
x_n  = \hat{x}_{n-1} + z_n \label{eq:periergia}.
\end{equation}
Then \eqref{eq:bary_update} reads
\begin{equation}
\hat{x}_{n} -\hat{x}_{n-1}   = 
\frac{\e^{-\nu f(x_n)}}{m_{n-1} + \e^{-\nu f(x_n)}}z_n \label{eq:periergia2}.
\end{equation}

\section{MAIN RESULT: DESCENT PROPERTY OF THE BARYCENTER SEARCH} \label{sec:theorems}


{A randomized version} of the barycenter algorithm can be studied using  formula \eqref{eq:periergia2}.
If we consider $\hat{x}_n$ to be our best guess, on the basis of the information provided by the tests up to $x_n$, of where the minimum of $f(\cdot)$ might be found, then in the absence of any extra knowledge it makes sense to pick the curiosity $z_n$ as a random variable with some judiciously chosen probability distribution. 
In light of the central limit theorem, we will analyze the case where $z_n$ is  normal.

   With the goal of obtaining approximate formulas for the barycenter update rule, in the remainder of this section we  will assume  that $f(\cdot)$ is twice continuously differentiable with respect to the argument $x$. Differentiability is not required for the barycenter method to be useful, and this assumption can be weakened using Theorem~\ref{thm:under-noise}. 
   Define $F_n(z) = \frac{\e^{-\nu f(\hat{x}_{n-1} + z)}}{m_{n-1} + \e^{-\nu f(\hat{x}_{n-1} + z)}}$, and 
for subsequent use write
\begin{multline*}
\bar{F}_n (z) = 
  \frac{m_{n-1} \e^{-\nu f(\hat{x}_{n-1} + z)}}{(m_{n-1} + \e^{-\nu f(\hat{x}_{n-1} + z)})^2} \\= 
  \frac{m_{n-1} }{m_{n-1} + \e^{-\nu f(\hat{x}_{n-1} + z)}} F_n ,
  \end{multline*}
 so that $\frac{\partial F}{\partial z} = - \nu \bar{F} \frac{\partial f}{\partial z}.$ Here and in the computations that follow the subscript indicating dependence of the sample ordinality $n$ is omitted if there is no ambiguity.

\begin{theorem} \label{thm:average}
If $z_n$ has a Gaussian  distribution, the expected value of $\Delta\hat{x}_n = \hat{x}_n - \hat{x}_{n-1}$ is proportional to the average value of the gradient of $f(\hat{x}_{n-1} + z_n)$ in the support of the distribution of $z$.
\end{theorem}

\begin{proof}
The  claim is established by a  calculation, which in an appropriate notation is straightforward. Consider the probability density function
\[ 
p(z)= \frac{1}{\sqrt{(2\pi)^n |\Sigma|}} e^{- \frac{1}{2}(z- \bar{z})^T\Sigma^{-1}(z- \bar{z})}.
\]
Then $\frac{\partial p}{\partial z^\beta} = - \Sigma^{-1}_{\beta \alpha } (z^\alpha- \bar{z}^\alpha) p(z) $ so $z^\alpha p =  \bar{z}^\alpha p - \Sigma^{\alpha \beta} \frac{\partial p}{\partial z^\beta}$. Einstein's implicit summation of components with equal upper and lower indices convention is in force, with upper Greek indices for the components of $z$ and of $\Sigma$, and lower indices for the components of $\Sigma$'s inverse. 

With $\mathcal X$  the $n$--dimensional set where the curiosity $z$ takes its values, for each component $z^\alpha$ of the vector $z$ we have
\begin{multline*}
E \left [ F(z) z^\alpha \right]  = \int_{\mathcal X} F(z) z^\alpha p(z) \, \d z \\
= \int_{\mathcal X} F(z) \bar{z}^\alpha p(z) \, \d z 
- \Sigma^{\alpha \beta} \int_{\mathcal X} F(z)  \frac{\partial p}{\partial z^\beta}(z) \, \d z.
\end{multline*}
Using  integration--by--parts,
\begin{multline*}
 \int_{\mathcal X} F(z)  \frac{\partial p}{\partial z^\beta} \, \d z 
 +  \int_{\mathcal X} \frac{\partial F}{\partial z^\beta}   p(z) \, \d z \\
 =  \int_{\partial \mathcal X} F(z) p(z) \left (\frac{\partial}{\partial z^\beta} \contract \d z \right),
\end{multline*}
where $\frac{\partial}{\partial z^\beta} \contract \d z$ is an $(n-1)$--form which can be integrated at the boundary $\partial {\mathcal X}$ of the $n$--dimensional set $ {\mathcal X}$. The right--hand side is zero, because $F$ is bounded and  $p(z)$ vanishes at the border of $\mathcal X$, which is at infinity, hence
\begin{multline*}
E \left [ F(z) z^\alpha \right] 
= \bar{z}^\alpha \int_{\mathcal X} F(z)  p(z) \, \d z 
+\Sigma^{\alpha \beta} \int_{\mathcal X} \frac{\partial F}{\partial z^\beta} p(z) \, \d z\\
= E \left [ F(z)  \right] \bar{z}^\alpha 
-\nu  \Sigma^{\alpha \beta} \int_{\mathcal X} \bar{F} (z) \frac{\partial }{\partial z^\beta}\left( f(\hat{x}_{n-1} + z)\right)  p(z)\,  \d z,
\end{multline*}
so 
\begin{equation}
\label{eq:approx-grad}
\boxed
{
E \left [ \Delta\hat{x}_n \right] = E \left [ F_n (z)  \right]  \bar{z} -\nu  \Sigma \, E\left[ \bar{F_n}(z) \nabla f(\hat{x}_{n-1} + z) \right],
}
\end{equation} 
where the 2nd term is proportional to the negative gradient of $f$ as claimed. 
\end{proof}

Formula \eqref{eq:approx-grad} is a key result concerning the barycenter method. It shows that roughly speaking a random search performed in conjunction with the barycenter algorithm  follows   the direction of the negative average  gradient of the function to be minimized, the weighted average being taken over the domain where the search is performed. For a given $\nu$, the step size is essentially given by   $\Sigma$. Depending on the shape of the function $f$, large values of the variance of $z$ may compromise the descent property of the search method.

The term in $\bar{z}$ can be employed to incorporate extra knowledge in several manners. For example, if at each step we take $\bar{z} = \xi \; \Delta\hat{x}_{n-1}$, then the gradient term is responsible for the rate of change, or acceleration, of the search process. The factor $0 < \xi < 1$ is chosen to dampen oscillations and prevent instability. The case $\xi = 0$ corresponds to the garden--variety, non--accelerated gradient--like search.
With a view towards  intended applications to adaptive controller and filter tuning, here the author cannot resist citing his seldom--read paper \cite{acelera-letters} discussing tuners that set the 2nd derivative of the adjusted parameters, rather than the 1st derivative, as is more often done in the literature, as well as an application to filtering theory \cite{max-pait-jojoa}, where the 2nd difference is used in a discrete--time version.

The designer has the freedom to choose the free parameters $\nu$, $\Sigma$, and $\bar{z}$ of the randomized barycenter search procedure in order to achieve the most desirable convergence properties. 
{Understanding of the variance of}  $\Delta\hat{x}_n $, which depends on  the Hessian $\nabla^2 f$ of $f$, is  useful in picking these free parameters. 

\section{FURTHER PROPERTIES OF THE BARYCENTER SEARCH} \label{sec:further}

In this section we state 3 further results concerning the barycenter method. The proofs  are available in the arXiv \cite{2018bary-arXiv}.

Near a  minimum of a locally convex function,  the variance of the adjustment step grows less than linearly with the variance of the curiosity; the higher the Hessian and the larger $\nu$ is, the smaller the variance. This is a desirable property of the method, because it indicates that the barycenter moves around less than the test points. 

\begin{theorem} \label{thm:variance}
Under the conditions of Theorem~\ref{thm:average} and assuming that the variance of $z$ is small, the variance of $\Delta\hat{x}_n$ for $\bar{z} = 0$ near a critical point of $f(x)$ where $\nabla f= 0$ is approximately 
\begin{equation}
\label{eq:approx-variance}
\mathrm{Var} (\Delta \hat{x}) \approx 
 \Sigma E[F^2 ] - 2\nu \Sigma^T E\left[  F  \bar{F} \nabla^2 f \right] \Sigma.
\end{equation}
\end{theorem}
 \vspace{.5ex}
 
 The \emph{complex barycenter} is defined  term--by--term, for each component $x^\alpha$ of vector $x$, by the same formula as the barycenter in \eqref{eq:barycenter}, with a complex exponent $\nu$:
 \begin{equation}
\label{eq:barycenter-i}
{\eta}^\alpha_n = \frac{\sum_{i=1}^n x^\alpha_i \e^{-\nu f(x_i)}} {\sum_{i=1}^n \e^{-\nu f(x_i)}},
\end{equation}
but now our estimate of the extremum point is
 \begin{equation}
\label{eq:barycenter-modulus}
\hat{x}^\alpha_n = | \eta_n^\alpha |.
\end{equation}
In these formulas all $x_i \geq 0$. 
The algorithm is suggested by Feynman's interpretation of quantum electrodynamics \cite{feynman-qed,feynman2012quantum} and by the stationary phase approximation \cite{bleistein1975asymptotic,hormander2012analysis} used in the  asymptotic analysis of integrals. 

\begin{theorem} \label{thm:complex-interference}
The expected contribution of measurements made outside of any region where $\nabla f \approx 0$ is discounted by one factor, proportional to $\nabla f$ and to the ratio between the complex magnitude of $\nu$ and its real part, for each dimension of the search space.
\end{theorem}
 
This destructive interference, so to speak, between repeated measurements near points which are \textsl{not} candidates for minimizers is the justification for employing complex values of $\nu$.

Oftentimes each measurement of the function $f$ at point $x_i$ is corrupted by noise or experimental errors. In this case we still would like to minimize $f$, but now using oracle answers $f(x_i) + w_i$. For the purpose of analyzing the effect of noise on the results of the barycenter method, we consider the sequence $x_i$ as given and $w_i$ as an ergodic random process. A more elaborate analysis of the effect of noise on the sequence $\{x_i\}$ itself, which would depend on the recursive search algorithm used to generate the oracle queries, isn't  done here.

Define the nominal or ``noise--free'' values 
$\bar{m} =  \sum_{i=1}^n \e^{-\nu f(x_i)}$ and 
$\bar{\eta} = { \sum_{i=1}^n x_i \e^{-\nu f(x_i)}}/ \bar{m}$, and also the scalar quantity $\bar{\bar{m}} = \sum_{i=1}^n \e^{-2\nu f(x_i)}$, the vector quantity $\bar{\bar{\eta}} = \sum_{i=1}^n x_i \e^{-2\nu f(x_i)}/ \bar{\bar{m}}$, and the matrix quantity $\breve{\eta} = \sum_{i=1}^n x_i x_i^T \e^{-2\nu f(x_i)} / \bar{\bar{m}}$.
\begin{theorem} \label{thm:under-noise}  Assuming that $\sigma$ is small, under the circumstances above the mean and variance of $\eta$ can be expressed approximately as follows:
\begin{align}
\label{eq:mean-eta}
E[\eta]  & \approx  \bar{\eta}  +
\frac{ \bar{\bar{m}} }{\bar{m}^2 } (\bar{\eta} - \bar{\bar{\eta}})  \nu^2 \sigma^2 
\text{ and } \\
\label{eq:var-eta}
\var[\eta] & \approx 
\frac{ \bar{\bar{m}} }{\bar{m}^2 }(\bar{\eta} \bar{\eta}^T   -  \bar{\eta} \bar{\bar{\eta}} - \bar{\bar{\eta}} \bar{\eta} + \breve{\eta})
 \nu^2 \sigma^2. 
\end{align}
\end{theorem}
 \vspace{1ex}

What we learn from somewhat involved formulas~\eqref{eq:mean-eta}
 and \eqref{eq:var-eta} is that noise or measurement errors generate a bias and a variance in the barycenter, both proportional in 1st approximation to the variance of the noise. These  effects conspire to pull the barycenter away from the looked--after minimum of the function. However one may derive a measure of comfort in that the unwelcome errors tend to zero as the noise becomes smaller. 
Theorem~\ref{thm:under-noise} indicates that there is little reason to fear that the method breaks down under moderate measurement or computing errors.

The situation where the (unknown) function to be optimized is not smooth can be studied as a particular case of the optimization of $f(x) + w$, where $f(\cdot)$ itself is smooth, and $w(x)$ is the difference between the function under consideration and its smooth approximation, plus a noise or error parcel when applicable. 
In this case the assumption that $w_i$ is uncorrelated with $x_i$ can be objected to. On the other hand any function can be approximated with arbitrary precision by a smooth function, making 
the analysis reasonable   in the practical case when   
 the approximation error is overshadowed by  measurement or  numerical errors.

\section{ILLUSTRATIVE SIMULATIONS} \label{sec:simula}

In this section we will choose $z$ with mean  $ \xi \; ( \hat{x}_n - \hat{x}_{n-1})$ and variance $\sigma^2$. The idea is that the random term is responsible for the rate of change, or acceleration, of the search process. The factor $0 < \xi < 1$ is chosen to dampen oscillations and prevent instability. 

Figure~\ref{fig:banana-search1} depicts the functional values over number of evaluations of a search for the minimum of  Rosenbrock's nonconvex banana function $f(x) = 100 \left(x^2-y\right)^2+(1-x)^2$ using the randomized barycenter method. 
Figure~\ref{fig:banana-search2} shows the test points in search space and the evolution of the barycenter, starting from the initial guess and randomly sliding down the level curves towards the minimum at $(1,1)$. The randomized barycenter search parameters were chosen $\nu = 4, \xi = 0.6$, and the variance $\sigma$ starts at 2 and decreases linearly to .4 over the 80 sample tests.

The instance of the search procedure is typical of many tests performed, although of course the results of any nondeterministic search may vary, and the search could conceivably  fail for an unlucky finite number of samples. We have not made any effort to optimize the search, but the results are roughly speaking comparable to the more effective methods tested in, for example, \cite{derivative-free-book}, page 11. 

\begin{figure}[!htb]
\centering
\includegraphics[width=\linewidth]{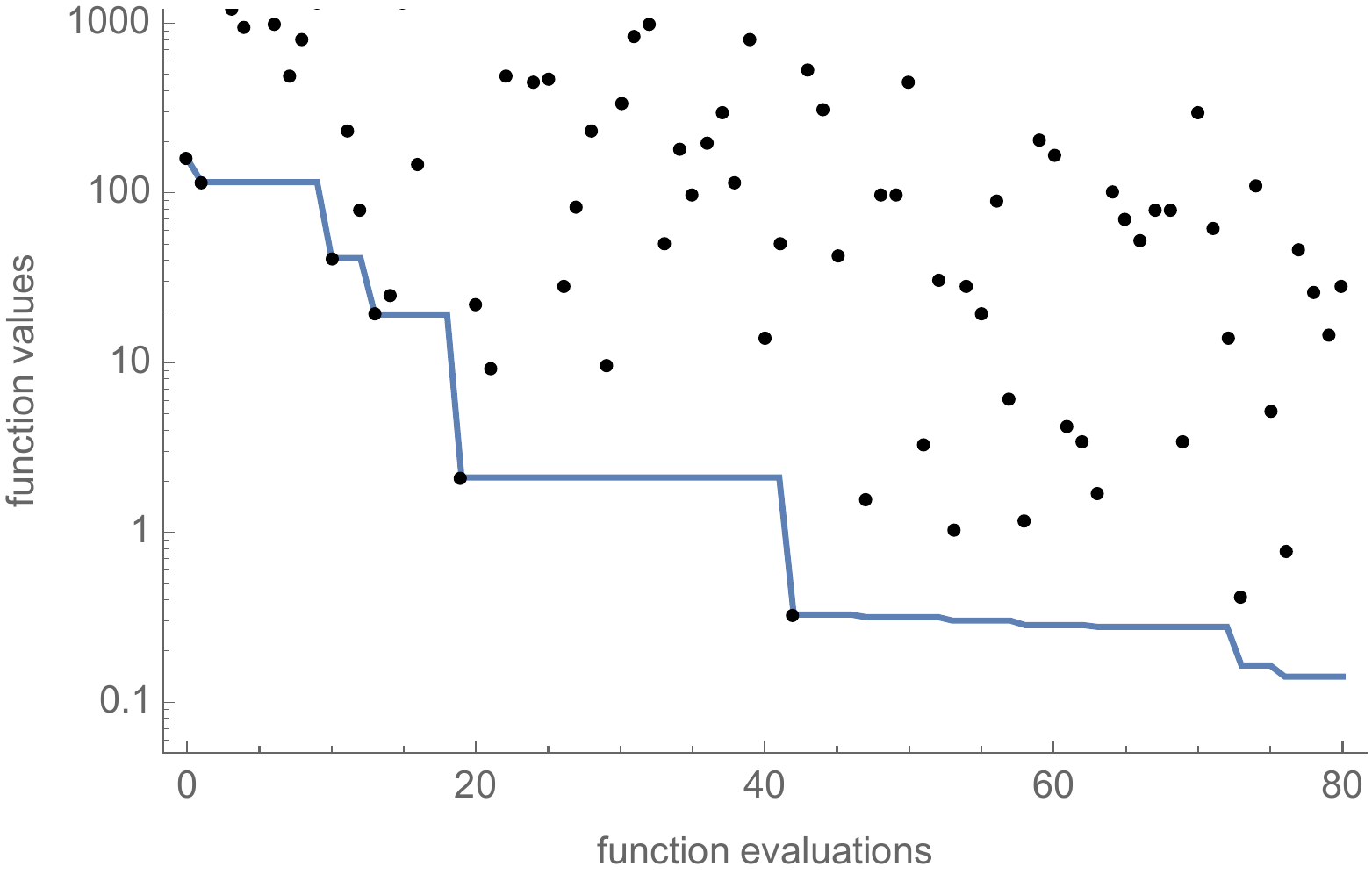}
\caption{Points show values of the banana function at the test points $x$, line shows function values at the estimates $\hat{x}$.}
\label{fig:banana-search1}
\end{figure}

\begin{figure}[!htb]
\centering
\includegraphics[width=\linewidth]{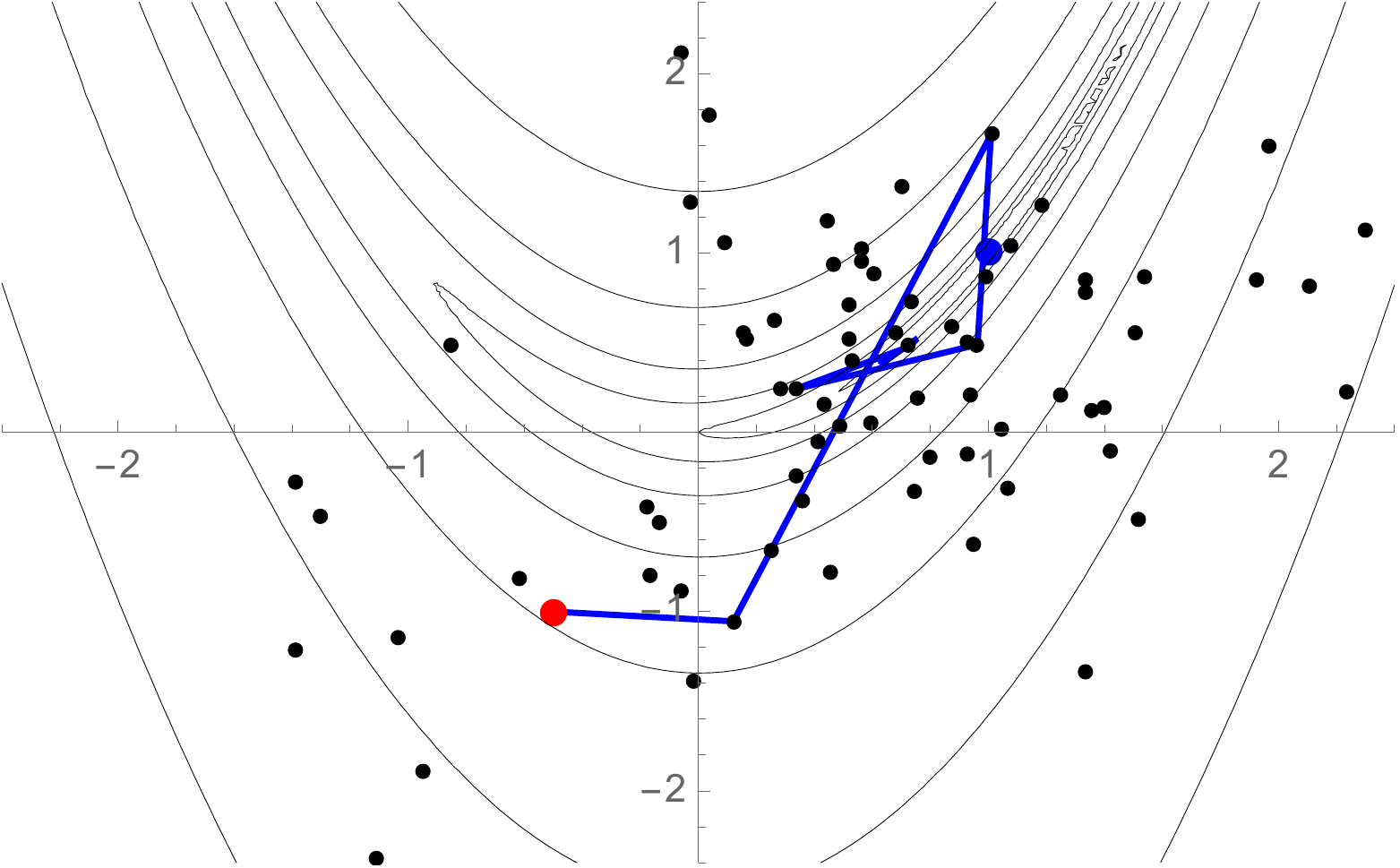}
\caption{Points show location of tests, line shows barycenter of previous points across level curves of banana function.}
\label{fig:banana-search2}
\end{figure}

%

Another test, shown analogously in Figures~\ref{fig:perturbed-search1} and \ref{fig:perturbed-search2}, concerns the perturbed quadratic function 
\begin{multline*}
10 x^2 \left(1 + \frac{75}{100} \frac{\cos (70 x)}{12}\right) + \frac{\cos (100  x)^2}{24}  \\
+ 
2 y^2 \left(1 + \frac{75}{100} \frac{\cos (70 y)}{12}\right) + \frac{\cos (100 y)^2}{24} + 4 x y,
\end{multline*}
also mentioned in \cite{derivative-free-book}. Although minimization of the function with high--frequency sinusoidal perturbations poses challenges to some derivative--free algorithms, the barycenter method performs comparably to the more effective of them. The parameters for the test that involved 100 function value estimations were $\nu = 1, \xi = 8/10$, and  variance $\sigma$ starting at 24/10 and decreasing with $966/1000$  at each  sample test. 

\begin{figure}[!htb]
\centering
\includegraphics[width=\linewidth]{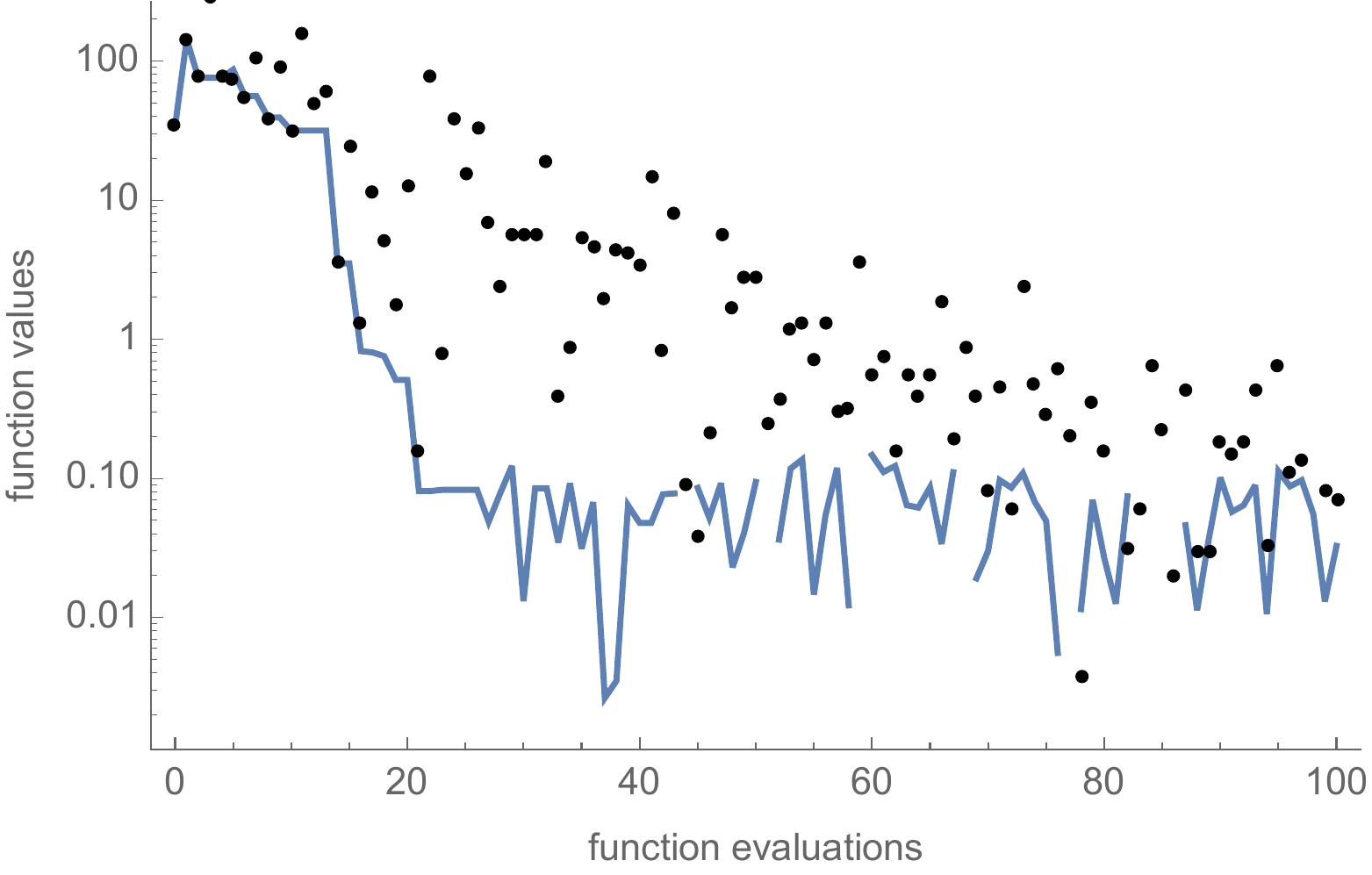}
\caption{Values of the perturbed quadratic function at $x$ and $\hat{x}$. Values missing from log plot are negative.}
\label{fig:perturbed-search1}
\end{figure}

\begin{figure}[!htb]
\centering
\includegraphics[width=\linewidth]{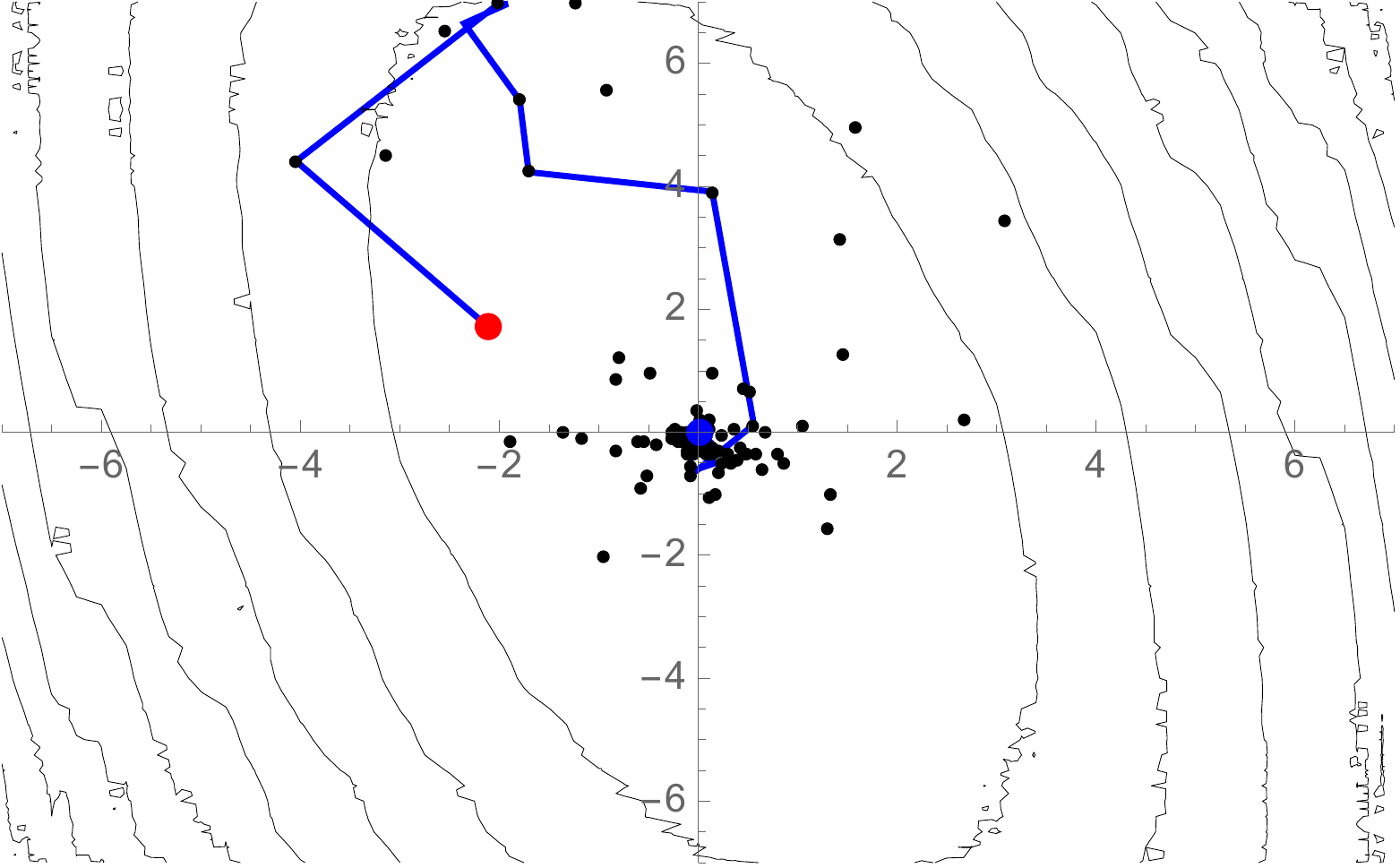}
\caption{Evolution of search for minimum of perturbed quadratic function.}
\label{fig:perturbed-search2}
\end{figure}

Figures~\ref{fig:canoe-search1} and \ref{fig:canoe-search2}  illustrate the barycenter search for the minimum of the ``canoe'' function $(1-e^{||x||^2}) \max ( ||x-c||^2,||x-d||^2)$, with $c = - d = [30, 40]^T$. This function was introduced as a benchmark to test mesh adaptive direct search  algorithms in page 209 of \cite{mesh-adaptive-canoe-audet}, as
it might present a challenge to derivative--based optimization methods (both 1st and 2nd order) and also to generalized pattern search (GPS) because of lack of differentiability. The derivative--free method under consideration is applicable to non--differentiable functions, and gives satisfactory results with  parameters $\nu = 8/10, \xi = 9/10$, and  variance $\sigma$ starts at 1 and decreases with $(982/1000)^n$  over 300 sample tests.

In the tests, the barycenter method parameters were chosen by the author with basis on aesthetic and didactic considerations, without the benefit of so--called  ``graduate student descent.'' Notice that some searches are nonmonotonic. A reliably monotonic search would be a sure indicator of lack of robustness to model noise or variability. Simulations were performed with exact  values in Mathematica\texttrademark, and were  fast in the low--dimensional problems studied. 

%
%

\begin{figure}[!htb]
\centering
\includegraphics[width=\linewidth]{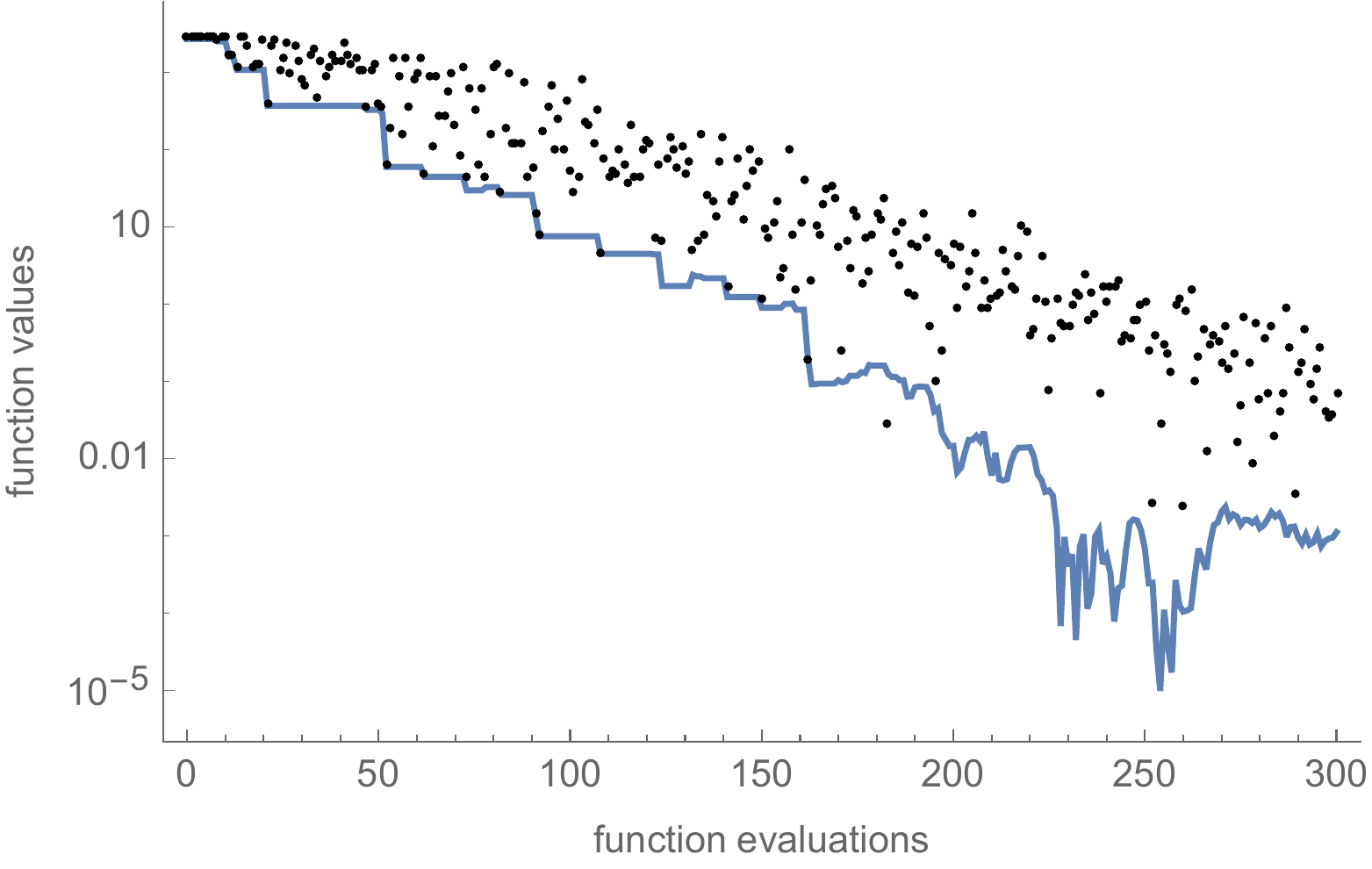}
\caption{Canoe function values at  test points and estimates.}
\label{fig:canoe-search1}
\end{figure}

\begin{figure}[!htb]
\centering
\includegraphics[width=\linewidth]{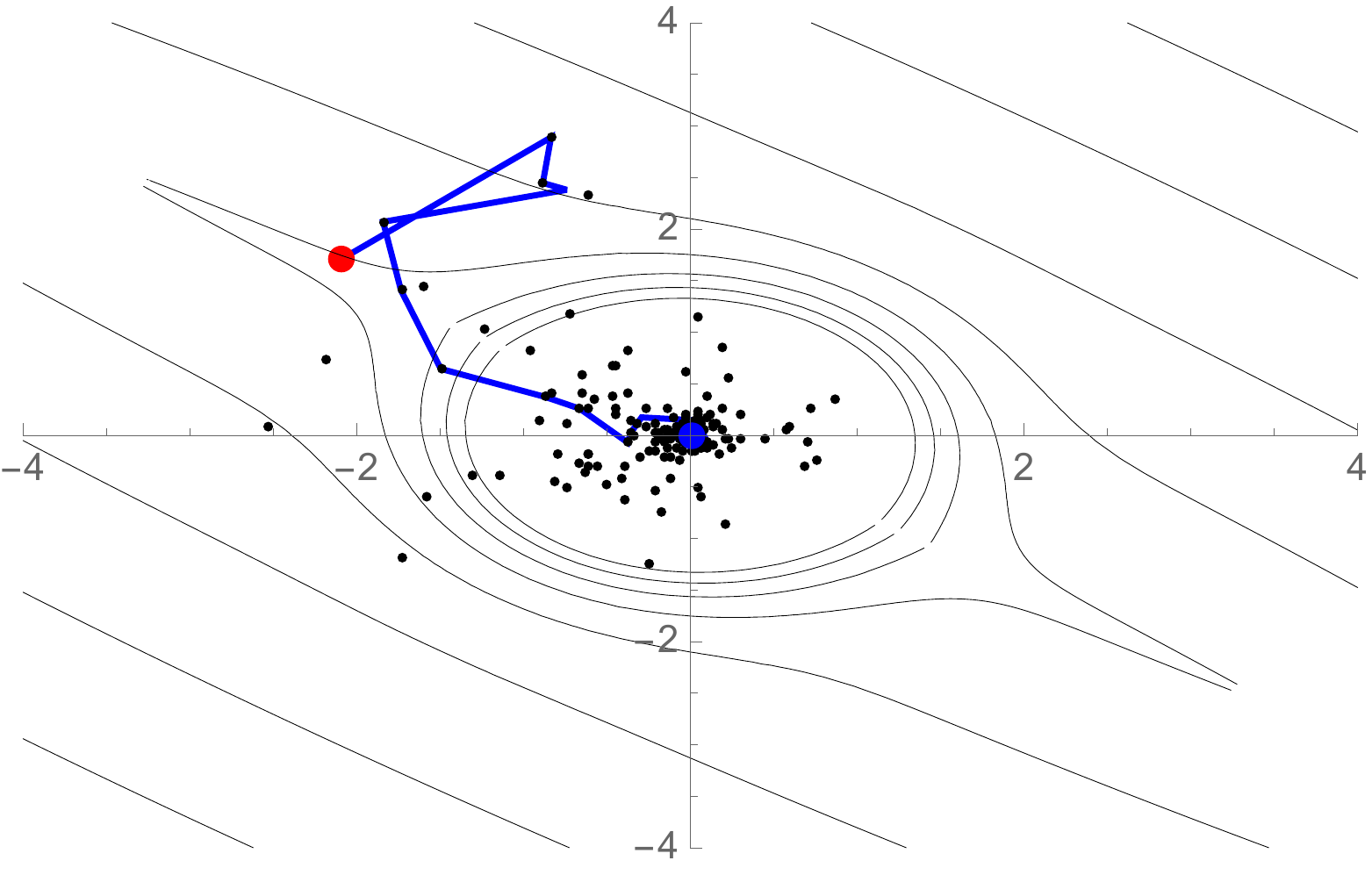}
\caption{Evolution of tests across level curves of canoe function.}
\label{fig:canoe-search2}
\end{figure}

\section{CONCLUSIONS AND FUTURE WORKS} \label{sec:discussion}

The theoretical properties of the method, stated in Sections~\ref{sec:theorems} and \ref{sec:theorems}, and demonstrated in full in \cite{2018bary-arXiv}, have been illustrated via simulations. The method shows itself to be robust and have acceptable performance in 2--dimensional benchmark problems. The barycenter method is  easy to code, and can be modified to adapt to  requirements, incorporating previous knowledge of a specific problem under study.  Performance could be optimized by using anything smarter than the purely random search we employed, the method's properties being guaranteed by the equivalence between the recursive and batch formulas in Section~\ref{sec:recurse-and-batch}. The method is also naturally parallelizable. Tests of a parallel version, as well as of the complex version presented in equations \eqref{eq:barycenter-i}-- \eqref{eq:barycenter-modulus} with properties given in Theorem~\ref{thm:complex-interference}, remain to be performed. The method is already in use for practical applications along the lines discussed in \cite{MOLI-transactions} and \cite{RomanoCDC:2014}. 

\subsection{Application to direct adaptive control}

I would like to end this paper sketching an application of the barycenter method to direct adaptive control. Given a control system with input $u$ and output $y$, one can design a dynamic feedback control law parameterized by some vector $\theta$, using standard tools from linear system theory. 
Assuming that there exists a parameterized controller that stabilizes the system, a form of \textsl{direct} adaptive control  consists in searching for  such a value of $\theta$,  without the intermediate step of estimating a system model.

This could be effected by trying to minimize a certain performance goal. For the sake of concreteness consider a familiar quadratic cost of the form
\[
J_i(\theta_i) = \int_{t_{i-1}}^{t_{i}} y^T Q y + u^T R u \, \d t,
\]
computed over an interval $[{t_{i-1}}, t_{i})$ during which the control applied is a constant feedback  with parameter $\theta_{i}$. 
In the linear--quadratic framework, it is possible to construct the feedback controller \cite{design-direct} such that the cost function is convex in $\theta$, up to terms depending on noise and perturbations. 

The control tuning problem at hand is nothing but a direct optimization problem, with the additional complication that the goal function is time--varying besides depending on $\theta$. The  controlled process works as a $0$--order oracle: when queried it supplies the value of the function to be optimized, but not of its 1st or 2nd derivatives. The sequence of feedback parameters $\theta_i$ and the corresponding costs $J_i(\theta_i)$ play the roles, respectively,  of the test points $x_i$ and the function $f(x_i)$ in the previous sections.  

A direct control algorithm would, at each instant $t_i$, pick a parameter vector $\theta_{i+1}$, which specifies the  controller that will be used to close the feedback loop during the subsequent interval $[{t_{i}}, t_{i+1})$. The choice of $\theta_{i+1}$ has to be a compromise between 2 goals: exploration of the parameter space, and convergence towards a minimum of  the cost function.  The main challenge is obtaining sufficiently fast convergence of the cost while applying controllers which can drive the system to unstable modes. This precludes the use of search methods with prespecified sequence of controllers. The properties of the barycenter method make it suitable to such an application. 

An advantage of the barycenter method is the possibility of combining directly tuned feedback gains with controllers designed via system identification tools. This would be a form of so--called data--driven, on--the--fly reconfiguring of a controller. It is very much along the lines of a combined direct + indirect adaptive controller synthesis.




\onecolumn

\bibliographystyle{IEEEtranS}

\bibliography{/Users/pait/paraguassu/latex/bibtex/bary,%
/Users/pait/paraguassu/latex/bibtex/acelera,%
/Users/pait/paraguassu/latex/bibtex/MIMOsysid}


\vfill


\begin{figure*}[!b]
\centering 
{\large The author is deeply grateful for João Cabrera's friendship \\
and for our conversations about control theory and  \\
Brazilian soccer. His memory shall be a blessing. \\}

\vspace{5ex}

\includegraphics[width=7cm]{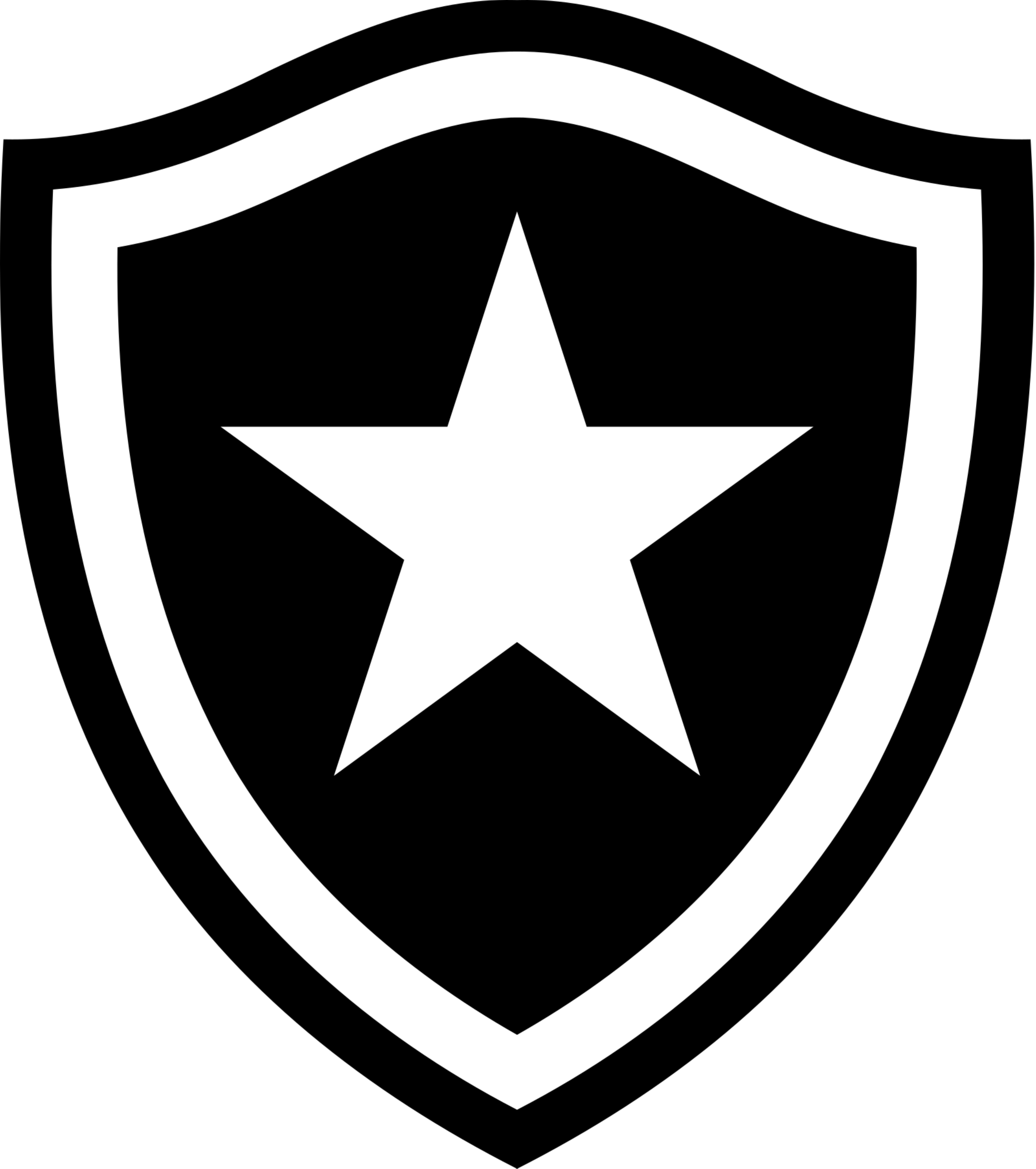}
{ \vspace{2ex}
 \\
\large{Tu és o Glorioso, \\
Tua Estrela Solitária te conduz.}}
\label{fig:botafogo}
\end{figure*}


\end{document}